\documentclass[a4paper,twoside,12pt]{article}
\usepackage[english]{babel}

\usepackage{amsmath,amsfonts,amssymb,amsthm}
\newtheorem{teo}{Theorem}
\newtheorem{lem}{Lemma}

\newtheorem{rem}{Remark}

 \title{{\bf   Convolution operators via orthogonal polynomials}}

\author{Maksim \,V.~Kukushkin   \\ \\
 \small  \textit{Moscow State University of Civil Engineering, 129337,  Moscow, Russia}\\
 \small\textit{Kabardino-Balkarian Scientific Center, RAS, 360051,  Nalchik, Russia}\\
\textit{\small\textit{kukushkinmv@rambler.ru}} }

\date{}
\begin{document}

\maketitle

\begin{abstract}
In this paper we aim to  generalize    results obtained    in the framework of fractional calculus by the way of reformulating them   in terms  of operator theory. In its own turn,   the achieved generalization allows us to spread the obtained technique on practical problems  that connected with various physical -
chemical processes.

\end{abstract}
\begin{small}\textbf{Keywords:} Positive operator; fractional power of an operator; semigroup generator;
 strictly accretive property.
\end{small}

\section{Introduction}

	The foundation  of   models  describing various physical -
chemical processes can be obtained by virtue of fractional calculus methods, the central point of which   is a concept of the Riemann-Liouville operator acting in the weighted Lebesgue space.
   In its own turn, the operator theory methods play an important role in  applications and need not any of special advertising. Having forced by these reasons,  we deal with mapping theorems for  operators acting on Banach spaces in order to obtain afterwards the desired results applicable to   integral operators.
   We also note that our interest was inspired by lots of previously known  results related to mapping theorems for fractional integral  operators      obtained by mathematicians  such as
    Rubin B.S. \cite{firstab_lit: Rubin},\cite{firstab_lit: Rubin 1},\cite{firstab_lit: Rubin 2}, Vakulov B.G. \cite{firstab_lit: Vaculov},   Samko S.G.  \cite{firstab_lit: Samko M. Murdaev},\cite{firstab_lit: Samko Vakulov B. G.}, Karapetyants N.K.
\cite{firstab_lit: Karapetyants N. K. Rubin B. S. 1},\cite{firstab_lit: Karapetyants N. K. Rubin B. S. 2}.

In this paper we offer one method of studying the Sonin operator \cite{Sonin}. We claim  the  existence and uniqueness   theorem  formulated in terms of the Jacoby series coefficients which  gives us an opportunity to find and classify a solution of the Sonin-Abel equation  due to an asymptotic of the  right side.
  Let us remind that the so called mapping theorem  for the Riemann-Liouville operator (the particular case of the Sonin operator)  were firstly studied   by H. Hardy and Littlewoode \cite{firstab_lit H-L1}  and nowadays is known as the Hardy-Littlewood theorem with limit index. However there was an attempt to extend this theorem on some class of  weighted Lebesgue spaces defined as functional spaces endowed with the following norm
  $$
  \|f\|_{L_{p}(I\!,\,\beta,\gamma)}:=\|f\|_{L_{p}(I\!,\mu)} ,\,\mu(x)=\omega^{\,\beta,\gamma}(x):=(x-a)^{\beta}(b-x)^{\gamma},\,\beta,\gamma\in \mathbb{R},\,I:=(a,b).
  $$
  In this dirrection the mathematicians such as  Rubin B.S., Karapetyants N.K. \cite{firstab_lit: Karapetyants N. K. Rubin B. S. 1} and others
    had  success.  All these create the prerequisite to invent another approach for studying the Riemann-Liouville operator mapping properties  that was successfully applied  in the paper \cite{kukushkin2019axi}.

    Assume that the functions
    $$
    \varrho,\vartheta \in L_{1}(I'),\,I'=(0, b-a)\subset \mathbb{R}.
    $$
   are such that the so called Sonin condition holds
 $$
  \varrho \ast\vartheta =1.
$$
Consider the operators
 $$
 _{s}I^{\varrho}_{a+}\varphi (x):=\int\limits_{a}^{x}\varrho(x-t)f(t)dt,\,\varphi \in L_{1}(I,\beta,\gamma);
 $$
$$
_{s}D^{\vartheta}_{a+}f(x):=\frac{d}{dx}\int\limits_{a}^{x}\vartheta(x-t)f(t)dt,\; f\in\, _{s} I^{\varrho}_{a+} ( L_{1}(I,\beta,\gamma)).
$$
In an ordinary way (see\cite{firstab_lit:samko1987}), we prove that
$$
_{s}D^{\vartheta}_{a+}\, _{s} I^{\varrho}_{a+}\varphi=\varphi,\,\varphi \in L_{1}(I,\beta,\gamma).
$$
Consider the Abel-Sonin equation under most general assumptions on the right part
\begin{equation} \label{kuk-cond}
_{s}I^{\varrho}_{a+}\varphi(x):=\int\limits_{a}^{x}\varrho(x-t)\varphi(t)dt=f\in L_{1}(I,\beta,\gamma).
\end{equation}
We use the following notations for Jacobi polynomials and reflated expressions
 $$
p_{n}^{\, \beta,\gamma }(x)= \delta_{n} (x-a)^{-\beta}(b-x)^{-\gamma}\frac{d^{n}}{dx^{n}}\left[(x-a)^{\beta+n}(b-x)^{\gamma+n}\right],\,\beta,\gamma>-1,\,n\in \mathbb{N}_{0},
$$
where
$$
\delta_{n}  (\beta,\gamma) =\frac{(-1)^{n}}{(b-a)^{n+(\beta+\gamma+1)/2}}\cdot \sqrt{\frac{(\beta+\gamma+2n+1) \Gamma(\beta+\gamma+n+1)}{n!\Gamma(\beta +n+1)\Gamma( \gamma+n+1)}} \;,\,
$$
$$
\delta_{0} (\beta-1,\gamma-1)= \frac{\Gamma( \beta +\gamma )}{(b-a)^{ (\beta+\gamma-1)/2}},\,\delta'_{0} (\beta-1,\gamma-1)=  \Gamma( \beta +\gamma ),
$$
$$
\delta_{0} (\beta,\gamma) =  \frac{1}{\sqrt{\Gamma(\beta  +1)\Gamma( \gamma +1)}}\,,\;\beta+\gamma+1=0.
$$
$$
\tilde{C}_{n}^{k}(\beta,\gamma):= \sum\limits_{i=0}^{k}C^{i}_{n}  \tbinom{n+\beta} {n-i}   \tbinom{n+\gamma} {i}  C^{i}_{k}    \tbinom{n-i}{k-i}  i!.
$$
We also use the following notations 
$$
f_{n}(\beta,\gamma)=\int\limits_{a}^{b}f(x)p^{\,\beta,\gamma}_{n}(x)\omega^{\,\beta,\gamma}(x)dx,\;S_{n}f:=\sum\limits_{k=0}^{n}f_{n}p^{\,\beta,\gamma}_{n},\;    A^{-1}_{mn}:=\int\limits_{a}^{b}p^{\,\beta,\gamma}_{m}(x)\left(_{s}D^{\vartheta}_{a+}p^{\,\beta,\gamma}_{n}f\right)(x)\omega^{\,\beta,\gamma}(x)dx,
$$
and short-hand notations $f_{n}=f_{n}(\beta,\gamma),\,p_{n}=p^{\,\beta,\gamma}_{n},\, (_{s}I^{\varrho}_{a+}\varphi)_{n}(\beta,\gamma):=\varphi^{\varrho}_{n}(\beta,\gamma),$ if their meaning is quite clear.
We need the following auxiliary lemmas.
\begin{lem}\label{L1}
Let $ \vartheta \in L_{1}(I'),$
then for $\beta,\gamma>0,$ we have
\begin{equation}\label{3}
  \int\limits_{a}^{b}p^{\,\beta,\gamma}_{m}(x)\left(_{s}D^{\vartheta}_{a+}p_{n}\right)(x)\omega^{\,\beta,\gamma}(x)dx= C_{m}(\beta,\gamma)
 \int\limits_{a}^{b}p^{\,\beta-1,\gamma-1}_{m+1}(x)\left(_{s}I^{\vartheta}_{a+}p_{n}\right)(x)\omega^{\,\beta-1,\gamma-1}(x)dx,
 $$
 $$
 \;C_{m}(\beta,\gamma):=\frac{\delta'_{m}(\beta,\gamma)}{\delta'_{m+1}(\beta-1,\gamma-1)},\;m,n\in \mathbb{N}_{0}.
\end{equation}
\end{lem}
\begin{proof}
Using the formula
$$
 p_{n}(x)=\int\limits_{a}^{x}p'_{n}(\tau)d\tau+p_{n}(a),
$$
it is not hard to calculate
$$
\int\limits_{a}^{x}\vartheta(x-t)p^{\,\beta,\gamma}_{n}(t)  dt=p^{\,\beta,\gamma}_{n}(a)\int\limits_{a}^{x}\vartheta(x-t)   dt+\int\limits_{a}^{x}\vartheta(x-t)dt \int\limits_{a}^{t}p'_{n}(\tau)d\tau=
$$
$$
 =p^{\,\beta,\gamma}_{n}(a)\int\limits_{a}^{x}\vartheta(x-t)   dt+\int\limits_{a}^{x}p'_{n}(\tau)d\tau  \int\limits_{\tau}^{x}\vartheta(x-t)dt=
$$
$$
 =p^{\,\beta,\gamma}_{n}(a)\int\limits_{a}^{x}\vartheta(x-t)   dt+\int\limits_{a}^{x}p'_{n}(\tau)d\tau  \int\limits_{\tau}^{x}\vartheta(t-\tau)dt=
$$
$$
 =p^{\,\beta,\gamma}_{n}(a)\int\limits_{a}^{x}\vartheta(x-t)   dt+\int\limits_{a}^{x}dt  \int\limits_{a}^{t}\vartheta(t-\tau)p'_{n}(\tau)d\tau.
$$
Hence
\begin{equation}\label{4}
(_{s}D^{\vartheta}_{a+}p_{n})(x)=\frac{d}{dx}\int\limits_{a}^{x}\vartheta(x-t)p^{\,\beta,\gamma}_{n}(t)  dt= p^{\,\beta,\gamma}_{n}(a) \vartheta(x-a)     +   \int\limits_{a}^{x}\vartheta(x-t)p'_{n}(t)dt.
\end{equation}
Note that making change of   a variable  $x-t=b-\tau$, we have
$$
\int\limits_{a}^{b}\left|\int\limits_{a}^{x}\vartheta(x-t)p'_{n}(t)dt\right|dx\leq\int\limits_{a}^{b}|p'_{n}(t)|dt\int\limits_{t}^{b}|\vartheta(x-t)|dx=
 \int\limits_{a}^{b}|p'_{n}(t)|dt \int\limits_{t}^{b }|\vartheta(b-\tau)|d\tau<\infty.
$$
Therefore $_{s}D^{\vartheta}_{a+}p_{n}\in L_{1}(I).$
Now the claimed result  can be established by virtue of integration by parts
$$
 \int\limits_{a}^{b}p^{\,\beta,\gamma}_{m}(x)\left(_{s}D^{\vartheta}_{a+}p_{n}\right)(x)\omega^{\,\beta,\gamma}(x)dx:=
 \int\limits_{a}^{b}p^{\,\beta,\gamma}_{m}(x) \frac{d}{dx}\, _{s}I^{\vartheta}_{a+}p_{n} (x)\omega^{\,\beta,\gamma}(x)dx=
$$
$$
 = \delta_{m}(\beta,\gamma)\int\limits_{a}^{b}\varphi^{(m)}_{m}(x) \frac{d}{dx}\, _{s}I^{\vartheta}_{a+}p_{n} (x) dx=-\delta_{m}(\beta,\gamma)\int\limits_{a}^{b}\varphi^{(m-1)}_{m}(x)  \, _{s}I^{\vartheta}_{a+}p_{n} (x) dx=
$$
$$
= -\frac{\delta_{m}(\beta,\gamma)}{ \delta_{m+1}(\beta-1,\gamma-1)}
\int\limits_{a}^{b}p^{\,\beta-1,\gamma-1}_{m+1}(x)\left(_{s}I^{\vartheta}_{a+}p_{n}\right)(x)\omega^{\,\beta-1,\gamma-1}(x)dx=
$$
$$
=  \frac{\delta'_{m}(\beta,\gamma)}{ \delta'_{m+1}(\beta-1,\gamma-1)}
\int\limits_{a}^{b}p^{\,\beta-1,\gamma-1}_{m+1}(x)\left(_{s}I^{\vartheta}_{a+}p_{n}\right)(x)\omega^{\,\beta-1,\gamma-1}(x)dx,\;m,n\in \mathbb{N}_{0}.
$$
\end{proof}
The following lemma that plays the principal role in extension of the previous  result.
\begin{lem}\label{L2}
Let $\vartheta\in L_{2}(I',-k,0),\,k>0 ,\,  \beta,\gamma>0,$  then the following estimate holds
$$
\left|\int\limits_{a}^{b} p_{m}^{\beta-1,\gamma-1}(x)  \, _{s}I_{a+}^{\vartheta}  f (x)\,\omega_{1}(x)   dx\right|\leq C\|f\|_{L_{2}(I,\beta,\gamma)},\; m=0,1,2,...,\,,
$$
where $ \omega_{1}(x)=(x-a)^{\beta-1}(b-x)^{\gamma-1}.$
\end{lem}
\begin{proof}
Let us consider the following reasonings
$$
\left|\int\limits_{a}^{b}\,\omega_{1}(x)  p_{m}^{\beta-1,\gamma-1}(x) \, _{s}I_{a+}^{\vartheta} f(x) dx\right|=
 \left|\int\limits_{a}^{b}f(t) dt  \int\limits_{t}^{b}  \vartheta(x-t) \omega_{1}(x) p_{m}(x) dx \right|\leq
 $$
 $$
 \leq\|f\|_{L_{2}(I,\beta,\gamma)} \left(\int\limits_{a}^{b} \omega^{-1}(t) \left|\int\limits_{t}^{b}  \vartheta(x-t)  \omega_{1}(x)p_{m}(x) dx\right|^{2}dt\right)^{1/2}=I_{1}.
$$
Using the generalized  Minkovskii inequality, we get
$$
I_{1} \leq\|f\|_{L_{2}(I,\beta,\gamma)}\int\limits_{a}^{b}p_{m}(x)\omega_{1}(x) \left(\int\limits_{a}^{x}\left| \vartheta(x-t)\right|^{2}\omega^{-1}(t) dt\right)^{1/2}dx\leq
$$
$$
\leq\|f\|_{L_{2}(I,\beta,\gamma)}\left(\int\limits_{a}^{b} \omega_{1}(x) dx\int\limits_{a}^{x}\left| \vartheta(x-t)\right|^{2}\omega^{-1}(t) dt\right)^{1/2}.
$$
Making the change of the variable twice, we have
$$
\int\limits_{a}^{b} \omega_{1}(x) dx\int\limits_{a}^{x}\left| \vartheta(x-t)\right|^{2}\omega^{-1}(t) dt=\int\limits_{a}^{b} \omega_{1}(x) dx\int\limits_{0}^{x-a}\left| \vartheta(t)\right|^{2}\omega^{-1}(x-t) dt=
$$
$$
=\int\limits_{0}^{b-a} \omega_{1}(x+a) dx\int\limits_{0}^{x}\left| \vartheta(t)\right|^{2}\omega^{-1}(x+a-t) dt=
\int\limits_{0}^{b-a}\left| \vartheta(t)\right|^{2}  dt\int\limits_{t}^{b-a}\omega^{-1}(x+a-t)\omega_{1}(x+a)dx =
$$
$$
=\int\limits_{0}^{b-a}\left| \vartheta(t)\right|^{2}t^{-k} t^{k}  dt\int\limits_{t}^{b-a}\omega^{-1}(x+a-t)\omega_{1}(x+a)dx .
$$
Note that, without lose of generality, we may make the following representation $(\delta_{1}+\delta_{2}=k)$
$$
t^{ k}\int\limits_{t}^{b-a}\omega_{1}(x+a)\omega^{-1}(x+a-t)dx=t^{-k}\int\limits_{t}^{b-a}x^{\beta-1}(b-a-x)^{\gamma-1} (x-t)^{-\beta} (b-a-x+t)^{-\gamma } dx=
$$
$$
 =t^{\delta_{1}+\delta_{2}}\int\limits_{t}^{b-a}x^{\beta-1}(b-a-x)^{\gamma-1} (x-t)^{-\beta} (b-a-x+t)^{-\gamma+ \delta_{2}}(b-a-x+t)^{- \delta_{2}} dx\leq
$$
$$
\leq C t^{\delta_{1}+\delta_{2}}\int\limits_{t}^{b-a}x^{\beta-1}(b-a-x)^{ -1+ \delta_{2}} (x-t)^{-\beta}  (b-a-x+t)^{- \delta_{2}} dx\leq
$$
$$
\leq C \int\limits_{t}^{b-a}x^{\beta-1+\delta_{1}}(b-a-x)^{ -1+ \delta_{2}} (x-t)^{-\beta}  t^{ \delta_{2}}(b-a-x+t)^{- \delta_{2}} dx\leq
$$
$$
\leq C\int\limits_{t}^{b-a} (b-a-x)^{ -1+ \delta_{2}} (x-t)^{-1+\delta_{1}}  t^{ \delta_{2}}(b-a-x+t)^{- \delta_{2}} dx\leq
$$
$$
\leq C \int\limits_{t}^{b-a} (b-a-x)^{ -1+ \delta_{2}} (x-t)^{-1+\delta_{1}}   dx=  B(\delta_{1},\delta_{2}) C.
$$
Combining these estimates, we obtain
$$
I_{1} \leq B(\delta_{1},\delta_{2})\|f\|_{L_{2}(I,\beta,\gamma)}\int\limits_{0}^{b-a}\left| \vartheta(t)\right|^{2} t^{l}dt.
$$
The last estimate proves the desired result.
\end{proof}

\begin{lem}\label{L3}
Let $\vartheta\in L_{2}(I' ),   \,  \beta,\gamma>0,$  then the following estimate holds
$$
   \|_{s}I_{a+}^{\vartheta}  f\|_{L_{2}(I,\beta,\gamma)}  \leq C\|f\|_{L_{2}(I,\beta,\gamma)}.
$$
\end{lem}
\begin{proof}
Using the generalized  Minkovskii inequality, we get
$$
\|_{s}I_{a+}^{\vartheta}  f\|_{L_{2}(I,\beta,\gamma)} \leq  \int\limits_{a}^{b}|f(t)|  \left(\int\limits_{t}^{b}\left| \vartheta(x-t)\right|^{2}\omega (x) dx\right)^{1/2}\!\!dt\leq
$$
$$
\leq\|f\|_{L_{2}(I,\beta,\gamma)}\left(\int\limits_{a}^{b} \omega(x) dx\int\limits_{a}^{x}\left| \vartheta(x-t)\right|^{2}\omega^{-1}(t) dt\right)^{1/2}= \|f\|_{L_{2}(I,\beta,\gamma)}\times I_{1},
$$
here $\omega(x):=\omega^{\,\beta,\gamma}(x).$
Repeating the reasonings of Lemma \ref{L2} and making the change of   variable twice,   we have
$$
I_{1}=\int\limits_{0}^{b-a} \omega (x+a) dx\int\limits_{0}^{x}\left| \vartheta(t)\right|^{2}\omega^{-1}(x+a-t) dt=
 \int\limits_{0}^{b-a}\left| \vartheta(t)\right|^{2}   dt\int\limits_{t}^{b-a}\omega^{-1}(x+a-t)\omega (x+a)dx.
$$
Note that
$$
 \int\limits_{t}^{b-a}\omega^{-1}(x+a-t)\omega (x+a)dx= \int\limits_{t}^{b-a}x^{\beta }(b-a-x)^{\gamma } (x-t)^{-\beta} (b-a-x+t)^{-\gamma } dx\leq
$$
$$
\leq  C \int\limits_{t}^{b-a}x^{\beta }(b-a-x)^{\gamma } (x-t)^{-\beta} (b-a-x )^{-\gamma } dx\leq
 C \int\limits_{t}^{b-a}   (x-t)^{-\beta}  dx,\,t\in(0,b-a ).
$$
 The last estimate proves the desired  result.
\end{proof}

\section{The main theorem}
Before formulating the main theorem, let us make the following  notations
$$
\mathfrak{B}^{ \beta ,\gamma}_{p}( f,\xi  ):=\sum\limits_{n=1}^{\infty}|f_{n}|^{p}n^{\xi },\,  f_{n}:=\int\limits_{a}^{b}f(x)p_{n}^{\,\beta,\gamma}(x)\omega(x)dx.
$$
Consider the Abel-Sonin equation under most general assumptions on the right-hand side
$$
_{s}I^{\varrho}_{a+}\varphi=f\in L_{2}(I,\beta,\gamma).
$$
We have the following theorem
\begin{teo}\label{T1}
Assume that $2\leq p<\infty,\,0<\beta,\gamma<1,$   the following conditions hold
\begin{equation}\label{5}
\mathfrak{B}^{ \beta-1,\gamma-1}_{p}( _{s}I^{\vartheta}_{a+} f,\xi )<\infty,\;\sum\limits_{m=0}^{\infty}   f^{\vartheta}_{m }  p^{\beta-1 ,\gamma-1 }_{m }(a)=0,
\end{equation}
where $ \xi=(5/2+\max\{\beta,\gamma\}) (p-2)+2,$
then there exists a unique solution  of the Abel-Sonin equation in $L_{p}(I,\beta,\gamma) $ represented by its series.
Moreover,  in the case   $p=2,\,0<\beta,\gamma<1,$ we claim that   conditions \eqref{5} are necessary, so we have a criterion.
\end{teo}
\begin{proof}
The sufficiency part of existence:
Using   formula \eqref{3},   we obtain
\begin{equation}\label{6}
 \delta'_{m+1}(\beta-1,\gamma-1)\int\limits_{a}^{b}p^{\,\beta,\gamma}_{m}(x)\left(_{s}D^{\vartheta}_{a+}S_{k}f\right)(x)\omega^{\,\beta,\gamma}(x)dx=
$$
$$
= \delta'_{m}(\beta,\gamma)
\int\limits_{a}^{b}p^{\,\beta-1,\gamma-1}_{m+1}(x)\left(_{s}I^{\vartheta}_{a+}S_{k}f\right)(x)\omega^{\,\beta-1,\gamma-1}(x)dx,\;k,m,n\in \mathbb{N}_{0}.
\end{equation}
Combining   Lemma \ref{L2}, Jacoby series  expansion, that is given by virtue of the fact $f\in L_{2}(I,\beta,\gamma),$    we can easily extend the last relation as follows
\begin{equation}\label{6.1.1}
\int\limits_{a}^{b}p^{\,\beta,\gamma}_{m}(x)\left(_{s}D^{\vartheta}_{a+}S_{k}f\right)(x)\omega^{\,\beta,\gamma}(x)dx\rightarrow
$$
$$
\rightarrow\frac{\delta'_{m}(\beta,\gamma)}{\delta'_{m+1}(\beta-1,\gamma-1)}
\int\limits_{a}^{b}p^{\,\beta-1,\gamma-1}_{m+1}(x)\left(_{s}I^{\vartheta}_{a+} f\right)(x)\omega^{\,\beta-1,\gamma-1}(x)dx,\;k\rightarrow\infty,\,m,n\in \mathbb{N}_{0}.
\end{equation}
It is clear that we can rewrite   last relation in the following form
$$
  \left|\sum\limits_{n=0}^{\infty}A^{-1}_{mn}f_{n}\right|=  C_{m}   \left|\int\limits_{a}^{b}p^{\,\beta-1,\gamma-1}_{m }(x)\left(_{s}I^{\varrho}_{a+}  f\right)(x)\omega^{\,\beta-1,\gamma-1}(x)dx\right|,
$$
where
$$
C_{m}=\frac{\delta'_{m}(\beta,\gamma)}{ \delta'_{m+1}(\beta-1,\gamma-1)}= \sqrt{(m+1)(\beta+\gamma+m)} .
$$
Thus, due to the theorem conditions we have
$$
\sum\limits_{m=1}^{\infty}\left|\sum\limits_{n=0}^{\infty}A^{-1}_{mn}f_{n}\right|^{p}m^{\xi-p }\leq C\, \mathfrak{B}^{ \beta-1,\gamma-1}_{p}( _{s}I^{\varrho}_{a+} f,\xi )<\infty.
$$
Let us calculate
$$
\xi-p =(5/2+\max\{\beta,\gamma\}) (p-2)+2-p=(1/2+\max\{\beta,\gamma\}) (p-2)+p-2.
$$
It implies that
$$
\sum\limits_{m=1}^{\infty}\left|\sum\limits_{n=0}^{\infty}A^{-1}_{mn}f_{n}\right|^{p}M^{p-2}_{m} m^{p-2 }<\infty,\;M_{m}=m^{1/2+\max\{\beta,\gamma\}}.
$$
Having applied  the Zigmund-Marczincevich theorem, we get that there exists a function $\psi\in L_{p}(I,\beta,\gamma),$ such that
$$
\sum\limits_{n=0}^{\infty}A^{-1}_{mn}f_{n}=\psi_{m},\,m\in  \mathbb{N}_{0}.
$$
As the consequences, we have
\begin{equation}\label{7}
\int\limits_{a}^{b}p^{\,\beta,\gamma}_{m}(x)\left(_{s}D^{\vartheta}_{a+}S_{k}f\right)(x)\omega^{\,\beta,\gamma}(x)dx\rightarrow \int\limits_{a}^{b}p^{\,\beta,\gamma}_{m}(x)\psi(x)\omega^{\,\beta,\gamma}(x)dx,\,m\in \mathbb{N}_{0},\,k\rightarrow\infty;
\end{equation}
\begin{equation}\label{7.0}
 \int\limits_{a}^{b}p^{\,\beta,\gamma}_{m}(x)\psi(x)\omega^{\,\beta,\gamma}(x)dx =C_{m}\int\limits_{a}^{b}p^{\,\beta-1,\gamma-1}_{m+1}(x)\left(_{s}I^{\vartheta}_{a+} f\right)(x)\omega^{\,\beta-1,\gamma-1}(x)dx.
\end{equation}
Using  simple reasonings, it is not hard to calculate the following formula
$$
  \int\limits_{a}^{b}p^{\,\beta ,\gamma }_{m }(x) S_{t}\psi(x)\omega^{\,\beta,\gamma}(x)dx     =C_{m}\int\limits_{a}^{b}p^{\,\beta-1,\gamma-1}_{m+1}(x)\omega^{\,\beta-1,\gamma-1}(x)dx \int\limits_{a}^{x}S_{t}\psi(t)dt=
$$
$$
  =  C_{m}\int\limits_{a}^{b}p^{\,\beta-1,\gamma-1}_{m+1}(x)\left(_{s}I^{\vartheta}_{a+}\, _{s}I^{\varrho}_{a+}S_{t}\psi\right)(x)\omega^{\,\beta-1,\gamma-1}(x)dx,\,
 t=0,1,...,\,.
$$
Now, using       the  Jacoby series expansion for the function $\psi$ (it is possible by virtue of the Zigmund-Marczincevich theorem) and  Lemmas \ref{L2},\ref{L3} we can extend the previous relation as follows
$$
   C_{m}\int\limits_{a}^{b}p^{\,\beta-1,\gamma-1}_{m+1}(x)\left(_{s}I^{\vartheta}_{a+}\, _{s}I^{\varrho}_{a+} \psi\right)(x)\omega^{\,\beta-1,\gamma-1}(x)dx=
  \int\limits_{a}^{b}p^{\,\beta ,\gamma }_{m }(x)  \psi(x)\omega^{\,\beta,\gamma}(x)dx.
$$
Combining this relation with \eqref{7.0}, we get
$$
\int\limits_{a}^{b}p^{\,\beta-1,\gamma-1}_{m+1}(x)\left(_{s}I^{\vartheta}_{a+}\, _{s}I^{\varrho}_{a+}\psi\right)(x)\omega^{\,\beta-1,\gamma-1}(x)dx=\int\limits_{a}^{b}p^{\,\beta-1,\gamma-1}_{m+1}(x)\left(_{s}I^{\vartheta}_{a+} f\right)(x)\omega^{\,\beta-1,\gamma-1}(x)dx,
$$
$$
\,m=0,1,...,\,.
$$
It implies that
\begin{equation}\label{8.0}
_{s}I^{\vartheta}_{a+}\, _{s}I^{\varrho}_{a+}\psi=\!_{s}I^{\vartheta}_{a+} f+\tilde{C} \;\mathrm{a.e.}\,,
\end{equation}
where
$$
\tilde{C}=\int\limits_{a}^{b}p^{\,\beta-1,\gamma-1}_{0}\left\{_{s}I^{\vartheta}_{a+} f -  \, _{s}I^{\vartheta}_{a+}\, _{s}I^{\varrho}_{a+}\psi \right\} \omega^{\,\beta-1,\gamma-1}(x)dx.
$$
Analogously to the reasonings given above, having applied Lemmas \ref{L2},\ref{L3}  it is not hard to establish the following equality
$$
\int\limits_{a}^{b}p^{\,\beta ,\gamma }_{m }(x) \,_{s} I^{\varrho}_{a+}\psi \,\omega^{\,\beta ,\gamma }(x) dx=C _{m}\int\limits_{a}^{b}p^{\,\beta-1,\gamma-1}_{m+1}(x) \,_{s}I^{\varrho}_{a+}\,_{s}I^{\vartheta}_{a+}\, _{s}I^{\varrho}_{a+}\psi \,\omega^{\,\beta-1,\gamma-1}(x) dx.
$$
Hence, using \eqref{8.0}  we obtain
$$
 \int\limits_{a}^{b}p^{\,\beta ,\gamma }_{m }(x) \,_{s} I^{\varrho}_{a+}\psi \,\omega^{\,\beta ,\gamma }(x) dx=
$$
$$
= C _{m}\int\limits_{a}^{b}p^{\,\beta-1,\gamma-1}_{m+1}(x)  \,  _{s}I^{\varrho}_{a+}\,_{s}I^{\vartheta}_{a+} f \,\omega^{\,\beta-1,\gamma-1}(x) dx+
  C _{m}\int\limits_{a}^{b}p^{\,\beta-1,\gamma-1}_{m+1}(x) \,_{s}I^{\varrho}_{a+}\tilde{C} \, \omega^{\,\beta-1,\gamma-1}(x)dx.
$$
 Using Lemmas \ref{L2},\ref{L3} in an absolutely analogous way,  we obtain
$$
C_{m}\int\limits_{a}^{b}p^{\,\beta-1,\gamma-1}_{m+1}(x)  \,  _{s}I^{\varrho}_{a+}\,_{s}I^{\vartheta}_{a+} f \,\omega^{\,\beta-1,\gamma-1}(x) dx =\int\limits_{a}^{b}p^{\,\beta ,\gamma }_{m }(x)  \,    f \,\omega^{\,\beta ,\gamma }(x) dx.
$$
 In additional, using  a trivial equality
$$
 \int\limits_{a}^{x}\varrho(x-t)dt= \int\limits_{a}^{x}\varrho(t-a)dt,
$$
we get
$$
\tilde{C}\int\limits_{a}^{b}p^{\,\beta ,\gamma }_{m }(x) \varrho(x-a) \, \omega^{\,\beta ,\gamma }(x)dx=C _{m}\int\limits_{a}^{b}p^{\,\beta-1,\gamma-1}_{m+1}(x) \,_{s}I^{\varrho}_{a+}\tilde{C} \, \omega^{\,\beta-1,\gamma-1}(x)dx.
$$
Therefore, combining the above results, we get
$$
\int\limits_{a}^{b}p^{\,\beta ,\gamma }_{m }(x) \,_{s} I^{\varrho}_{a+}\psi \,\omega^{\,\beta ,\gamma }(x) dx = \int\limits_{a}^{b}p^{\,\beta ,\gamma }_{m }(x)  \,    f \,\omega^{\,\beta ,\gamma }(x) dx+
 \tilde{C}\int\limits_{a}^{b}p^{\,\beta ,\gamma }_{m }(x) \varrho(x-a) \, \omega^{\,\beta ,\gamma }(x)dx .
$$
Thus, due to the basis property of Jacoby polynomials in $L_{2}(I,\beta,\gamma),$  we have
$$
 I^{\varrho}_{a+}\psi=f+\tilde{C}\,\varrho(x-a)\,a.e.
$$
Now let us express  the constant $\tilde{C}$ in terms of the theorem conditions.
Applying Lemmas \ref{L2},\ref{L3}  we can easily prove
$$
 \int\limits_{a}^{b}\omega^{\,\beta-1,\gamma-1}(x)  I^{1}_{a+}  S_{t}\psi  (x)dx=
 $$
 $$
 =\int\limits_{a}^{b}\omega^{\,\beta-1,\gamma-1}(x) _{s}I^{\vartheta}_{a+}  \,_{s}I^{\varrho}_{a+} S_{t}  \psi  (x)dx\rightarrow \int\limits_{a}^{b}\omega^{\,\beta-1,\gamma-1}(x) _{s}I^{\vartheta}_{a+}  \,_{s}I^{\varrho}_{a+}   \psi  (x)dx,\,t\rightarrow\infty.
$$
Therefore, we have
\begin{equation}\label{9.0}
 \int\limits_{a}^{b}p^{\,\beta-1,\gamma-1}_{0}\left\{  \, _{s}I^{\vartheta}_{a+} f- I^{1}_{a+}S_{t}\psi \right\} \omega^{\,\beta-1,\gamma-1}(x)dx\rightarrow \tilde{C},\,t\rightarrow\infty.
\end{equation}
Consider more precisely
$$
\int\limits_{a}^{x}    S_{t}\psi(t) dt=\sum\limits_{m=0}^{t}\psi_{m}\int\limits_{a}^{x}  p^{\,\beta ,\gamma }_{m }(t)    dt
  =\sum\limits_{m=0}^{t}\frac{ \psi_{m}\left( p^{\beta-1 ,\gamma-1 }_{m+1}(x) - p^{\beta-1 ,\gamma-1 }_{m+1}(a)\right)}{ \sqrt{(m+1)(\beta+\gamma+m)}  }.
$$
It follows that
$$
 \int\limits_{a}^{b}\omega^{\,\beta-1,\gamma-1}(x)  I^{1}_{a+}  S_{t}\psi  (x)dx=\int\limits_{a}^{b} \sum\limits_{m=0}^{t}\frac{ \psi_{m}\left( p^{\beta-1 ,\gamma-1 }_{m+1}(x) - p^{\beta-1 ,\gamma-1 }_{m+1}(a)\right)}{ \sqrt{(m+1)(\beta+\gamma+m)}  }\,\omega^{\,\beta-1,\gamma-1}(x)dx=
$$
$$
 =
 - \sum\limits_{m=0}^{t}\frac{ \psi_{m}    p^{\beta-1 ,\gamma-1 }_{m+1}(a) }{ \sqrt{(m+1)(\beta+\gamma+m)}  }\,\int\limits_{a}^{b}\omega^{\,\beta-1,\gamma-1}(x)dx=
  - B(\beta ,\gamma )(b-a)^{\beta+\gamma-1}\sum\limits_{m=0}^{t}\frac{ \psi_{m}    p^{\beta-1 ,\gamma-1 }_{m+1}(a) }{ \sqrt{(m+1)(\beta+\gamma+m)}  }  =
  $$
  $$
 =- B(\beta ,\gamma )(b-a)^{\beta+\gamma-1}\sum\limits_{m=0}^{t}\frac{C_{m} f^{\vartheta}_{m+1}    p^{\beta-1 ,\gamma-1 }_{m+1}(a) }{ \sqrt{(m+1)(\beta+\gamma+m)}  }=
 - B(\beta ,\gamma )(b-a)^{\beta+\gamma-1}\sum\limits_{m=0}^{t}   f^{\vartheta}_{m+1}    p^{\beta-1 ,\gamma-1 }_{m+1}(a).  \,
$$
Hence
$$
 \int\limits_{a}^{b}p^{\,\beta-1,\gamma-1}_{0}\left\{ \, _{s}I^{\vartheta}_{a+} f   -  I^{1}_{a+}S_{t}\psi \right\} \omega^{\,\beta-1,\gamma-1}(x)dx=
 B(\beta ,\gamma )(b-a)^{\beta+\gamma-1}p^{\,\beta-1,\gamma-1}_{0}\sum\limits_{m=0}^{t}   f^{\vartheta}_{m+1}    p^{\beta-1 ,\gamma-1 }_{m+1}(a) +f^{\vartheta}_{0}=
$$
$$
= \sqrt{B(\beta ,\gamma )}(b-a)^{(\beta+\gamma-1)/2} \sum\limits_{m=0}^{t}   f^{\vartheta}_{m+1}    p^{\beta-1 ,\gamma-1 }_{m+1}(a) +f^{\vartheta}_{0}=
$$
$$
= \sqrt{B(\beta ,\gamma )}(b-a)^{(\beta+\gamma-1)/2} \sum\limits_{m=0}^{t}   f^{\vartheta}_{m }    p^{\beta-1 ,\gamma-1 }_{m }(a)  \rightarrow C,\,t\rightarrow\infty.
$$

The necessity part of existence: Now assume that there exists a solution of the Abel-Sonin equation in $L_{2}(I,\beta,\gamma), \,0<\beta,\gamma<1.$ Then
using Lemmas \ref{L2},\ref{L3} we can easily establish (analogously to the above) the following equality
$$
 \int\limits_{a}^{b}p^{\,\beta,\gamma}_{m}(x)\psi(x)\omega^{\,\beta,\gamma}(x)dx =C_{m}\int\limits_{a}^{b}p^{\,\beta-1,\gamma-1}_{m+1}(x)\left(_{s}I^{\vartheta}_{a+}\,_{s}I^{\varrho}_{a+} \psi\right)(x)\omega^{\,\beta-1,\gamma-1}(x)dx=
$$
$$
=C_{m}\int\limits_{a}^{b}p^{\,\beta-1,\gamma-1}_{m+1}(x)\left(_{s}I^{\vartheta}_{a+} f\right)(x)\omega^{\,\beta-1,\gamma-1}(x)dx,\, m\in \mathbb{N}_{0}.
$$
Hence
$$
\mathfrak{B}^{ \beta-1,\gamma-1}_{2}( _{s}I^{\vartheta}_{a+} f,2 )<\infty.
$$
Since $\psi$ is a solution,  then  \eqref{8.0} is fulfilled, where $\tilde{C}=0.$ We  can also establish \eqref{9.0}, in the way that   was used  above.
Having repeated the above reasonings, we come to the relation
$$
\sum\limits_{m=0}^{\infty}   f^{\vartheta}_{m }    p^{\beta-1 ,\gamma-1 }_{m }(a)=0.
$$

The proof of uniqueness:  
Assume  that there  exists a solution $\psi$   and another solution $\phi$ in $L_{p}(I,\beta,\gamma)$ of the Sonin-Abel equation,  and let us   denote $\xi:=\psi-\phi.$
Denote
$$
I_{n}:=\left(a+ \frac{1}{n},b- \frac{1}{n}\right),
$$
then the following assumptions     are fulfilled
$$
 \bigcup\limits_{n=1}^{\infty}I_{n}=I,\;I_{n}\subset I_{n+1},\,\mu (I\!\setminus\!I_{n})\rightarrow 0,\,n\rightarrow \infty, \; L_{p}(I,\beta,\gamma)\subset L_{p}(I_{n} ).
$$
The verification is left to a reader. In terms of these denotations,  it is also clear that
$$
\bigcup\limits_{n=1}^{\infty}C_{0}^{\infty}(I_{n})=C_{0}^{\infty}(I ),\;C_{0}^{\infty}(I_{n})\subset C_{0}^{\infty}(I_{n+1}).
$$
Let us show that
\begin{equation}\label{8}
 \forall\eta \in C_{0}^{\infty}(\Omega ),\,\forall\xi\in L_{p}(I,\beta,\gamma) ,\,\exists h\in L_{p'}(I,\beta,\gamma):
 $$
 $$
 \int\limits_{a}^{b}\xi(x)\eta(x)dx=\int\limits_{a}^{b}\xi(x)\,_{s}I^{\varrho}_{b-}\omega h(x)dx.
\end{equation}
It is not hard to prove that $_{s}D^{\vartheta}_{b-}\eta(x)\in C(\bar{I}),$ the proof is left to a reader. Moreover, we have the following.
Consider
$$
    \omega^{-1}(x) \,_{s}D^{\vartheta}_{b-}\eta(x)= (x-a)^{-\beta}(b-x)^{1-\gamma }(b-x)^{-1}\!\! _{s}D^{\vartheta}_{b-}\eta(x).
$$
Let us show that $_{s}D^{\vartheta}_{b-}\eta(b)=0.$   In accordance with the reasonings applied to obtain  formula \eqref{4}, we analogously get
$$
(_{s}D^{\vartheta}_{b+}\eta)(x) = \eta(b) \vartheta(b-x)     -   \int\limits_{x}^{b}\vartheta(t-x)\eta' (t)dt=    -\int\limits_{x}^{b}\vartheta(t-x)\eta' (t)dt.
$$
It is clear that
$$
\left|\int\limits_{x}^{b}\vartheta(t-x)\eta' (t)dt\right|\leq C\int\limits_{x}^{b}|\vartheta(t-x)| dt=\int\limits_{0}^{b-x}|\vartheta(t)| dt.
$$
Since $\vartheta\in L_{2}(I',-\varepsilon),$ then
$$
\int\limits_{0}^{b-x}|\vartheta(t)| dt=0,\,x=b.
$$
Hence we obtain the desired result i.e. $_{s}D^{\vartheta}_{b-}\eta(b)=0.$ We can also get without any difficulties, by using the previous results, the following relation
  $$
  \frac{d}{dx} \,_{s}D^{\vartheta}_{b+}\eta (x) =     \int\limits_{x}^{b}\vartheta(t-x)\eta'' (t)dt,
  $$
and it is clear that
$$
  \frac{d}{dx} \,_{s}D^{\vartheta}_{b+}\eta (x)   =0,\,x=b.
  $$
Now it gives us the following
$$
(b-x)^{-1}\!\! _{s}D^{\vartheta}_{b-}\eta(x)=(b-x)^{-1}\left\{_{s}D^{\vartheta}_{b-}\eta(x)-\,_{s}D^{\vartheta}_{b-}\eta(b)  \right\}\rightarrow 0,\,x\rightarrow b.
$$
Hence the function  $
\omega^{-1} D^{\alpha}_{b-}\eta
$
belongs to $L_{p'}(I,\beta,\gamma),$ if $\beta<1/(p'-1)$ (in particularly it is fulfilled  if $1<p'\leq 2$).  It implies that we have a representation $D^{\alpha}_{b-}\eta=\omega h,$ where $h$ belongs   to $L_{p'}(I,\beta,\gamma).$   By virtue of the fact $\eta\in C_{0}^{\infty}(I),$ we can easily prove a relation $ _{s}I^{\varrho}_{b-}\,_{s} D^{\vartheta}_{b-}\eta=\eta$ (a reader only ought to repeat the proof given above corresponding to the polynomial case).      Hence
$$
\eta=\,_{s} I^{\varrho}_{b-}\omega h,\,h\in L_{p'}(I,\beta,\gamma).
$$
 Taking into account the reasonings given above,   we get formula \eqref{8} which, on account of the Fubini theorem, can be rewritten as follows
$$
\int\limits_{a}^{b}\xi(x)\eta(x)dx=\int\limits_{a}^{b}\xi(x)\,_{s}I^{\varrho}_{b-}\omega h(x)dx=\int\limits_{a}^{b}\!_{s}I^{\varrho}_{a+}\xi(x)\, h(x)\omega(x)dx=0.
$$
Hence
$$
\int\limits_{I_{n}} \xi(x)\eta(x)dx  =0,\,\forall \eta \in C_{0}^{\infty}(I_{n}).
$$
We claim that  $\xi\neq 0.$ Therefore   in accordance  with the consequence of the Hahn-Banach theorem there exists the element $\varpi\in L_{p'}(I_{n}),$ such that
$$
 \left(\varpi,\xi \right)_{L_{2}(I_{n})}  =\|\psi-\phi\|_{ L_{p}(I_{n}) }>0.
$$
On the other hand, there exists  the sequence $\{\eta_{k}\}_{1}^{\infty}\subset C_{0}^{\infty}(I_{n}),$ such  that $\eta_{k} \rightarrow \varpi$ with respect to the norm $L_{p'}(I_{n}).$  Hence
$$
0=\left(\eta_{k},\xi \right)_{L_{2}(I_{n})} \rightarrow \left(\varpi, \xi\right)_{L_{2}(I_{n})}.
$$
  Hence $\psi=\phi$ almost everywhere  on the set $I_{n},\, n=1,2,...\,$ (the explanation of this reasoning is too simple   and can be found in any book devoted to Functional Analysis). In its own turn, it implies    that $\psi=\phi$ almost everywhere on the set $I.$        The    uniqueness has been proved. Thus the proof of sufficiency has been completed.

\end{proof}
\begin{rem}
In terms of Theorem \ref{T1} it seams to be easy to formulate    necessary conditions  of solvability of the Abel-Sonin equaton in $L_{p}(I,\beta,\gamma),\,0<\beta,\gamma<1,\,1<p<2.$
\end{rem}
For this purpose we need impose the additional conditions, first of them is
$
\varrho\in L_{p'}(I').
$
Due to this condition we can repeat the reasonings of Lemma \ref{L3} and prove
\begin{equation}\label{10.0}
   \|_{s}I_{a+}^{\varrho}  f\|_{L_{2}(I,\beta,\gamma)}  \leq C\|f\|_{L_{p}(I,\beta,\gamma)}.
\end{equation}
The second condition, we need impose, is the Polard condition
$$
 4\max\left\{\frac{\beta+1}{2\beta+3},\frac{\gamma+1}{2\gamma+3}\right\}<p< 4\min\left\{\frac{\beta+1}{2\beta+1},\frac{\gamma+1}{2\gamma+1}\right\}.
$$
After that our attention ought to be concentrate upon the following equality
$$
 \int\limits_{a}^{b}p^{\,\beta,\gamma}_{m}(x)\psi(x)\omega^{\,\beta,\gamma}(x)dx =C_{m}\int\limits_{a}^{b}p^{\,\beta-1,\gamma-1}_{m+1}(x)\left(_{s}I^{\vartheta}_{a+}\,_{s}I^{\varrho}_{a+} \psi\right)(x)\omega^{\,\beta-1,\gamma-1}(x)dx.
$$
It can be proved by applying Lemma \ref{L2}, estimate \eqref{10.0} and the basis property of   Jacoby polynomials.
Hence, having taken into account that $\psi$ is a solution, we get
$$
\int\limits_{a}^{b}p^{\,\beta,\gamma}_{m}(x)\psi(x)\omega^{\,\beta,\gamma}(x)dx=C_{m}\int\limits_{a}^{b}p^{\,\beta-1,\gamma-1}_{m+1}(x)\left(_{s}I^{\vartheta}_{a+} f\right)(x)\omega^{\,\beta-1,\gamma-1}(x)dx,\, m\in \mathbb{N}_{0}.
$$
Having applied the Zigmund-Maczincevich theorem, we have
$$
\sum\limits_{m=1}^{\infty}\left| \psi_{n}\right|^{p}M^{p-2}_{m} m^{p-2 }\leq  C \|\psi\|_{L_{p}(I,\beta,\gamma)},\;M_{m}=m^{1/2+\max\{\beta,\gamma\}}.
$$
Hence, by direct calculation, we obtain
$$
\mathfrak{B}^{ \beta-1,\gamma-1}_{p}( _{s}I^{\varrho}_{a+} f,\xi )<C\sum\limits_{n=1}^{\infty}|n\psi_{n}|^{p}n^{(5/2+\max\{\beta,\gamma\})(p-2)+2}=
$$
$$
 =C \sum\limits_{n=1}^{\infty}|\psi_{n}|^{p}n^{(3/2+\max\{\beta,\gamma\})(p-2)}<\|\psi\|^{p}_{L_{p}(I,\beta,\gamma)}<\infty.
$$
The prove of the fact
$$
\sum\limits_{m=0}^{\infty}   f^{\vartheta}_{m }    p^{\beta-1 ,\gamma-1 }_{m }(a)=0
$$ is absolutely analogous to the case $p=2.$ We should only repeat the scheme of the reasonings having taken into  account the additional conditions.
These reasonings can be formed in the following theorem.
\begin{teo}\label{T2}
Assume that there exists a solution  $\psi$  in $L_{p}(I,\beta,\gamma),\,0<\beta,\gamma<1,\,1<p<2$ of the Abel-Sonin equation,
  the following additional conditions hold
$$
\varrho\in L_{p'}(I'),\,4\max\left\{\frac{\beta+1}{2\beta+3},\frac{\gamma+1}{2\gamma+3}\right\}<p< 4\min\left\{\frac{\beta+1}{2\beta+1},\frac{\gamma+1}{2\gamma+1}\right\}.
$$
Then relations \eqref{5} hold.

\end{teo}

\end{document}